\documentclass[11pt,reqno]{amsart}
\usepackage{amssymb,amsmath,amsthm,epsfig,times}
\usepackage{mathrsfs}
\usepackage{epstopdf}
\usepackage{pdfpages}
\numberwithin{equation}{section} 
\usepackage{verbatim}
\usepackage[pdftex]{hyperref}

\def\C{{\mathbb C}}
\def\D{{\mathbb D}}

\def\P{{\mathbb P}}

\def\R{{\mathbb R}}
\def\*S{{\mathbb S}}

\def\Z{{\mathbb Z}}

\newcommand{\dis}{\displaystyle}
\newcommand{\cal}{\mathcal}

\newtheorem{theorem}{Theorem}
\newtheorem{proposition}{Proposition}[section]
\newtheorem{lemma}[proposition]{Lemma}
\newtheorem{corollary}[proposition]{Corollary}

\begin{document}
\title[Boundary integral formula for harmonic functions]
{Boundary integral formula for harmonic functions on Riemann surfaces.}

\author{Peter L. Polyakov}
\curraddr{}
\email{polyakov@uwyo.edu}

\subjclass[2010]{Primary: 30F, 32A10, 32A26}

\keywords{Riemann surfaces, Harmonic functions}

\begin{abstract}
We construct a boundary integral formula for harmonic functions on
smoothly-bordered subdomains of Riemann surfaces embeddable into $\C\P^2$. The formula
may be considered as an analogue of the Green's formula for domains in $\C$.
\end{abstract}

\maketitle

\section{Introduction.}\label{Introduction}

\indent
Let $\widetilde V$ be a Riemann surface
\begin{equation}\label{ProjectiveCurve}
{\widetilde V}=\left\{z\in \C\P^2:\ P(z)=0\right\}
\end{equation}
defined by the polynomial $P$ of degree $d$, and let
\begin{equation}\label{AffineCurve}
V=\left\{z\in {\widetilde V}: \varrho(z)<0\right\}=\widetilde{V}\setminus\bigcup_{r=1}^m V_r,
\end{equation}
be a subdomain of $\widetilde V$, where $\varrho$ is a smooth function on $\widetilde{V}$,
and $\left\{V_r\right\}_{r=1}^m$, $(m\leq d)$
is a collection of disjoint neighborhoods in $\widetilde V$ of the points at infinity
$$\left\{\zeta^{(1)},\dots,\zeta^{(d)}\right\}=\widetilde{V}\cap \left\{\zeta\in\C\P^2: \zeta_0=0\right\}.$$
By allowing inequality $m<d$ we allow the possibility of some neighborhoods $V_r$
to contain several points of the set ${\widetilde V}\cap \left\{z_0=0\right\}$.\\
\indent
The goal of the present article is the construction of a boundary integral formula
defining the values of a harmonic function $u$ on $V$ through the values of $u$ and of the form $\partial u$
on the boundary $bV$. In a sense the resulting formula may be considered as an analogue of the Green's
formula for a harmonic function $u$ and a solution of $\Delta_y G(x,y)=\delta_x$
in a domain $V\subset \C$
\begin{equation}\label{Green}
u(x)=\int_{bV}\left(u(y)\frac{\partial G(x,y)}{\partial\nu}-G(x,y)\frac{\partial u}{\partial\nu}(y)\right)ds(y).
\end{equation}\\
\indent
To construct the sought formula we use the formula from our earlier article \cite{P} for boundary
representation of holomorphic functions on open Riemann surfaces as in \eqref{AffineCurve}.
The statement of Theorem 1 from \cite{P}, where the holomorphic formula is proved, is included below in section~\ref{proof}.
The formula in \cite{P} is constructed as the residue of the formula on a tubular domain in the unit sphere
$\*S^{5}(1)=\left\{z\in\C^3: |z|=1\right\}$, namely
\begin{equation}\label{Uepsilon}
U^{\epsilon}=\left\{z\in \*S^{5}(1):\ \left|P(z)\right|<\epsilon, \varrho(z)<0\right\}.
\end{equation}
Therefore, its application requires two additional steps: extension of a holomorpic function from $V$
to a domain in $\C\P^2$, and further extension to some $U^{\epsilon}\subset \*S^{5}(1)$. The first
extension is constructed in Lemma~\ref{Extension} below, and the second is achieved, as in \cite{HP1},
by the identification of a function on a domain in $\C\P^2$ with its lift to a domain in $\*S^{5}(1)$
satisfying appropriate homogeneity conditions.\\
\indent
The motivation for the present work, though indirectly, came from the author's joint work with
Gennadi Henkin, who in the last years of his life became interested in an \lq\lq explicit\rq\rq\ solution
of the inverse problem of conductivity on Riemann surfaces, in which the conductivity function has to be reconstructed from the Dirichlet-to-Neumann map on its boundary (see \cite{C}, \cite{Ge}, \cite{HN},
and references therein).

\indent
Before formulating the main result of the article we introduce some additional objects and notations.
As in \cite{HP2} and \cite{P} we consider the Weil-type barrier \cite{WA} defined by
polynomials $\left\{Q^i(\zeta,z)\right\}_{i=0}^2$ satisfying:
\begin{equation}\label{HomogeneityConditions}
\left\{\begin{array}{ll}
P(\zeta)-P(z)=\sum_{i=0}^2Q^i(\zeta,z)\cdot\left(\zeta_i-z_i\right),\vspace{0.1in}\\
Q^i(\lambda\zeta,\lambda z)=\lambda^{d-1}\cdot Q^i(\zeta,z)\ \text{for}\ \lambda\in\C.
\end{array}\right.
\end{equation}
\indent
Another barrier, which was constructed in \cite{P}, is local with respect to $z$ and global
with respect to $\zeta$. It has the form
\begin{equation}\label{BoundaryBarrier}
F(w,\zeta)=\sum_{i=0}^2R_i(w)(\zeta_i-w_i),
\end{equation}
where we assume that for any point
$z\in V$ there exists a neighborhood ${\cal V}_z\ni z$ of $z$ and a
holomorphic vector-function $R(w)=\left(R_0(w),\ R_1(w),\ R_2(w)\right)\neq 0$ on
${\cal V}_z$ such that for $w\in {\cal V}_z$ the set
\begin{equation}\label{RCondition}
\begin{aligned}
&{\cal S}(w)=\left\{\zeta\in V: F(w,\zeta)=0\right\}\
\text{consists of finitely many points}\ \left\{w^{(0)}=w, w^{(1)},\dots, w^{(d-1)}\right\},\\
&\text{at which the line} \left\{\zeta: F(w,\zeta)=0\right\}
\text{transversally intersects}\ V.
\end{aligned}
\end{equation}
\indent
To construct a vector function $R(w)$ satisfying \eqref{RCondition} we use Bertini's
Theorem (see \cite{Ha}) and choose the vector $R(z)$ so that it satisfies condition
\eqref{RCondition} at the point $z$. Then we will have this condition satisfied for $w\in {\cal V}_z$
in a small enough ${\cal V}_z$.\\
\indent
For a set of points $\left\{w,w^{(1)},\dots,w^{(d-1)}\right\}$ such that function ${\dis \frac{w_1}{w_0} }$
takes distinct values at those points, we define matrix $A(w)$ as the following Vandermonde
$d\times d$ matrix
\begin{equation}\label{Vandermonde}
A(w)=\left[\begin{tabular}{cccc}
1&1&$\cdots$&1\vspace{0.05in}\\
${\dis \frac{w_1}{w_0} }$&${\dis \frac{w^{(1)}_1}{w^{(1)}_0} }$&$\cdots$&
${\dis \frac{w^{(d-1)}_1}{w^{(d-1)}_0} }$\vspace{0.05in}\\
$\vdots$&$\vdots$&&$\vdots$\vspace{0.05in}\\
${\dis \left(\frac{w_1}{w_0}\right)^{d-1} }$&
${\dis \left(\frac{w^{(1)}_1}{w^{(1)}_0}\right)^{d-1} }$&$\cdots$&
${\dis \left(\frac{w^{(d-1)}_1}{w^{(d-1)}_0}\right)^{d-1} }$
\end{tabular}\right].
\end{equation}
For a holomorphic function $g$ on $V$ we denote by $A_k[g](w,w^{(1)},\dots,w^{(d-1)})$
the matrix $A(w)$ with the $k$-th column replaced by the column
\begin{equation*}
\left[\begin{tabular}{c}
$G_0(w,w^{(1)},\dots,w^{(d-1)})$\\
$\vdots$\\
$G_{d-1}(w,w^{(1)},\dots,w^{(d-1)})$
\end{tabular}\right],
\end{equation*}
where
\begin{multline}\label{GValues}
G_k(w,w^{(1)},\dots,w^{(d-1)})
=\frac{2}{(2\pi i)^{3}}\Bigg(\sum_{j=0}^{d-1}
\lim_{\epsilon\to 0}\int_{\Gamma^{\epsilon}}
g(\zeta)\cdot\left(\frac{\zeta_1}{\zeta_0}\right)^k\\
\times\det\left[\frac{Q(\zeta,w^{(j)})}{P(\zeta)}\
\frac{R(w^{(j)})}{F(w^{(j)},\zeta)}\ \frac{\bar\zeta}{B(\zeta,w^{(j)})}\right]
d\zeta_0\wedge d\zeta_1\wedge d\zeta_2\Bigg),
\end{multline}
\begin{equation}\label{Gamma}
\Gamma^{\epsilon}=\left\{z\in \*S^{5}(1):\ \left|P(z)\right|=\epsilon, \varrho(z)=0\right\},
\end{equation}
$$B(\zeta,z)=\sum_{j=0}^2{\bar\zeta}_j\cdot\left(\zeta_j-z_j\right),$$
and barriers $P(\zeta)-P(z)$ and $F(w,\zeta)$ are defined in 
\eqref{HomogeneityConditions} and \eqref{BoundaryBarrier} respectively.

Below we formulate the main theorem of this article.

\begin{theorem}\label{Main} Let $V\subset \left\{\C\P^2\setminus\left\{z_0=0\right\}\right\}$
be as in \eqref{AffineCurve}, and let $u$ be a harmonic function on $V$.\\
\indent
Then there exist an $\epsilon>0$ and a holomorphic function $g$ on $U^{\epsilon}$, constructed
in \eqref{HolomorphicFunction} and Lemma~\ref{Extension} as an extension of a holomorphic function
on $V$ with the real part based on modification of $u$, and such that for an arbitrary $z\in V$,
neighborhood ${\cal V}_z$ of $z$ in $V$ satisfying conditions:
\begin{equation}\label{TwoConditions}
\left\{\begin{aligned}
&(i)\ \text{condition}\ \eqref{RCondition}\ \text{holds},\\
&(ii)\ \text{the function}\ {\dis \frac{w_1}{w_0} }\ \text{takes distinct values at the points}
\left\{w,w^{(1)},\dots,w^{(d-1)}\right\}\ \text{of}\ {\cal S}(w)\ \text{for}\ w\in {\cal V}_z,
\end{aligned}\right.
\end{equation}
and arbitrary $w\in {\cal V}_z$ the following equalities hold for the values of $u$ at the points
of ${\cal S}(w)$:
\begin{equation}\label{uValues}
u(w^{(k)})=\frac{1}{d+1}\cdot\mbox{Re}
\left\{\frac{w^{(k)}_0\cdot\det A_k[g](w,w^{(1)},\dots,w^{(d-1)})}{\det A(w)}\right\}
+\sum_{r=1}^{m-1}a_rh_r(w),
\end{equation}
where $\left\{h_r\right\}_1^{m-1}$ are fixed harmonic functions with log-type singularities at selected fixed
points in $\widetilde{V}\setminus V$, and the coefficients depending on the form $\partial u\Big|_{bV}$
are constructed in Proposition~\ref{ZeroPeriods-1}.
\end{theorem}

{\bf Remark 1.} The proof of Theorem~\ref{Main} is based on the application of the boundary value
formula for holomorphic functions on Riemann surfaces, constructed in \cite{P}. A connection with this
formula is established in Proposition~\ref{Periods} below. This proposition shows, in particular,
that on an open subset $V$ of a Riemann surface with finitely many boundary components
$bV=\cup_{r=1}^m \sigma_r$ every harmonic function $u\in C^{\infty}(\overline{V})$ is the real part
of a holomorphic function up to a finite-dimensional subspace, the result first obtained in \cite{WJ}.
Explicit construction in Proposition~\ref{Periods} shows in addition that the
codimension of the subspace of real parts of holomorphic functions in the space of all harmonic functions
on $V$ is equal to $m-1$, and depends only on the number of boundary components.

\indent
{\bf Remark 2.} The problem of existence of a holomorphic function with a given harmonic function as
its real part on a multiply connected domain in $\C$ or on a subdomain of a Riemann surface was addressed
earlier in \cite{KD, KS1,KS2} (see also relevant bibliography in those articles). In the case of a subdomain
of $\C$ the problem is easier, because the second set of equalities in \eqref{ZeroIntegrals}
is absent and Proposition~\ref{ZeroPeriods-2} below is not needed. In the articles above the construction
of the sought holomorphic function is based on the application of the Green kernel of $V$. Construction of
the Green kernel of an open Riemann surface is a delicate problem that to the best of author's knowledge
has no explicit solution, namely a solution, depending only on: 1) the defining equations of $\widetilde{V}$
from \eqref{GeneralProjectiveCurve}, and 2) equations defining contours
$\left\{\sigma_r\right\}_{j=1}^{m}$, or zero set of the function $\varrho$ from \eqref{GeneralAffineCurve}. 
Such problem is very close to the problem that we are solving in the present article, namely of finding
a boundary integral formula for harmonic functions on $V$ from \eqref{ProjectiveCurve}
and \eqref{AffineCurve}. Our proof of Proposition~\ref{Periods}, presented below, which relies on
Fourier analysis of functions on the fundamental region of $\widetilde{V}$ in $\C$ or in the unit disk,
allows us to construct an explicit (in the above mentioned sense) form - \eqref{uValues} - of an analogue
of the Green's identity on $V$.

{\bf Acknowledgments.} The author would like to thank Dima Khavinson for reading the manuscript and for
bringing the author's attention to articles \cite{KD, KS1,KS2}, where the problem of multivaluedness
of holomorphic functions with fixed real part on multiply connected domains and on Riemann surfaces
was addressed. The author also would like to thank the referee for suggestions improving the exposition
of results of the article.

\section{Modification of the original function.}
\label{uModification}

\indent
To prove the boundary representation formula \eqref{uValues} for a harmonic function $u$ on $V$ we will
use the real part of the Cauchy-type formula for holomorphic functions, which was constructed in \cite{P}.
However, not every harmonic function on $V$ is a real part of a holomorphic function
(see for example \cite{Fr}).
Therefore, we have to modify $u$ in such a way that the new harmonic function will be a real part
of a holomorphic function on $V$.
In the lemma below we give a necessary and sufficient condition for a harmonic function
on $V$ to be a real part of a holomorphic function.

\begin{lemma}\label{RealPart}
Let $u$ be a real-valued harmonic function on $V$. Then a holomorphic function
$f$ on $V$ admits the representation $f=u+iv$ with real-valued $v$, iff it satisfies
\begin{equation}\label{df}
df=2\partial u.
\end{equation}
\end{lemma}
\begin{proof}
We consider the differential form $\partial u$, which in a local coordinate system
of holomorphic coordinate
$\zeta=\xi+i\eta$ has the form
\begin{equation*}
\partial u=\frac{1}{2}(u_{\xi}-iu_{\eta})(d\xi+id\eta).
\end{equation*}
Since $u$ is harmonic, i.e. $\bar\partial\partial u=0$, $\partial u$ is a holomorphic
form on $V$. If there exists a holomorphic function $f=u+iv$, then we have 
\begin{multline*}
df=(\partial+\bar\partial)f=(\partial u+\bar\partial u)+i(\partial v+\bar\partial v)\\
=\frac{1}{2}(u_{\xi}-iu_{\eta})(d\xi+id\eta)+\frac{1}{2}(u_{\xi}+iu_{\eta})(d\xi-id\eta)
+\frac{i}{2}(v_{\xi}-iv_{\eta})(d\xi+id\eta)+\frac{i}{2}(v_{\xi}+iv_{\eta})(d\xi-id\eta)\\
=\frac{1}{2}(u_{\xi}+v_{\eta}+iv_{\xi}-iu_{\eta})(d\xi+id\eta)
+\frac{1}{2}(u_{\xi}-v_{\eta}+iu_{\eta}+iv_{\xi})(d\xi-id\eta)\\
=(u_{\xi}-iu_{\eta})(d\xi+id\eta)=2\partial u,
\end{multline*}
where we used the Cauchy-Riemann equations
\begin{equation*}
\begin{aligned}
&u_{\xi}=v_{\eta},\\
&u_{\eta}=-v_{\xi}.
\end{aligned}
\end{equation*}
\indent
On the other hand, if a function $f$ satisfies \eqref{df}, then $df$ is
a holomorphic differential form, i.e. $f$ is holomorphic and $u$, up to a constant, is its real part.
\end{proof}

\indent
In the proposition below
for $V \subset\widetilde{V}$ as in \eqref{AffineCurve}
we modify the given harmonic function $u$ on $V$ so that
the resulting function has zero integrals over the generators of $H^1(V,\Z)$.

\begin{proposition}\label{Periods}
Let $\widetilde V$ be a Riemann surface
\begin{equation}\label{GeneralProjectiveCurve}
{\widetilde V}=\left\{z\in \C\P^n:\ P_1(z)=\cdots=P_{n-1}(\zeta)=0\right\}
\end{equation}
defined by the polynomials $\left\{P_j\right\}_{j=1}^{n-1}$.
Let
\begin{equation}\label{GeneralAffineCurve}
V=\left\{z\in {\widetilde V}: \varrho(z)<0\right\}=\widetilde{V}\setminus\bigcup_{r=1}^m V_r,
\end{equation}
be a subdomain of $\widetilde V$, where $\varrho$ is a smooth function on $\widetilde{V}$,
and $\left\{V_r\right\}_{r=1}^m$, $\left(m\leq d=\prod_{j=1}^{n-1}\deg P_j\right)$
is a collection of disjoint neighborhoods in $\widetilde V$ with smooth curves $\sigma_r:=bV_r$
of the points at infinity
$$\left\{\zeta^{(1)},\dots,\zeta^{(d)}\right\}=\widetilde{V}\cap \left\{z\in\C\P^n: \zeta_0=0\right\}.$$
\indent
Let $u$ be a harmonic function on $V$, and let $\left\{\gamma_j\right\}_{j=1}^{2p}$
be a set of closed simple paths representing the generators of the group
$H_1(\widetilde{V},\Z)$.\\
\indent
Then there exists an explicit harmonic function $h$ on $V$, defined by the integrals
$\left\{\int_{\sigma_r}\partial u\right\}_{r=1}^m$, such that for the differential form
\begin{equation*}
\partial(u-h)=\frac{1}{2}\big((u-h)_{\xi}-i(u-h)_{\eta}\big)(d\xi+id\eta)
\end{equation*}
the following equalities hold
\begin{equation}\label{ZeroIntegrals}
\begin{aligned}
&\int_{\sigma_r}\partial(u-h)=0\hspace{0.1in}\mbox{for}\hspace{0.1in}r=1,\dots,m,\\
&\int_{\gamma_j}\partial(u-h)=0\hspace{0.1in}\mbox{for}\hspace{0.1in}j=1,\dots,2p.
\end{aligned}
\end{equation}
\end{proposition}

\begin{proof}
We divide the proof of Proposition~\ref{Periods} into two Propositions~\ref{ZeroPeriods-1}
and \ref{ZeroPeriods-2} below.

\begin{proposition}\label{ZeroPeriods-1}
\indent
Under conditions of Proposition~\ref{Periods} there exists a set of real-valued harmonic functions
$\left\{h_j\right\}_{j=1}^{m-1}\in C^1(\overline{V})$ such that for any harmonic function
$u\in C^1(\overline{V})$ there exist coefficients $\left\{a_j\right\}_{j=1}^{m-1}\in \C$ satisfying
the first set of equalities in \eqref{ZeroIntegrals}
\begin{equation}\label{aCoefficients}
\int_{\sigma_r}\partial\left(u-\sum_{j=1}^{m-1}a_jh_j\right)=0\hspace{0.1in}\text{for}
\hspace{0.1in} r=1,\dots,m.
\end{equation}
\end{proposition}
\indent
The proof of Proposition~\ref{ZeroPeriods-1} is based on the application of three lemmas below.
In those lemmas we use an O. Forster's idea from his book \cite{Fo} to consider two special cases of
construction of harmonic functions on open Riemann surfaces.

\begin{lemma}\label{FirstStep}
Let $z^{(1)}, z^{(2)} \in U\Subset \widetilde{V}$, where $U$ is a coordinate neighborhood in $\widetilde{V}$
with coordinate function $\zeta$ such that $\zeta(z^{(1)})=0, \zeta(z^{(2)})=1/2$.
Then there exist neighborhoods $U_j\ni z^{(j)}$ and a holomorphic function $f$ on
$\widetilde{V}\setminus\left(U_2\cup z^{(1)}\right)$ such that
\begin{equation}\label{LocalEquality}
f\big|_{U_1}=\zeta^{-1}\cdot\phi,
\end{equation}
where $\phi$ is a holomorphic function in $U_1$ with $\phi(z^{(1)})\neq 0$, and
\begin{equation}\label{IndicesEquality}
\left\{\begin{aligned}
&\frac{1}{2\pi i}\int_{bU_1}\partial\log{|f(\zeta)|^2}=-1,\\
&\frac{1}{2\pi i}\int_{bU_2}\partial\log{|f(\zeta)|^2}=1.
\end{aligned}\right.
\end{equation}
\end{lemma}
\begin{proof}
We consider a smooth function $\psi$ such that
\begin{equation}\label{PsiFunction}
\psi(\zeta)=\left\{\begin{aligned}
&1\ \text{if}\ |\zeta|<r_1\ \text{for some}\ r_1\in (3/4,1),\\
&0\ \text{if}\ |\zeta|>r_2\ \text{for some}\ r_2\in(r_1,1),
\end{aligned}\right.
\end{equation}
and define function $f_0$ on $U\setminus z^{(1)}$ by the formulas
\begin{equation}\label{fZero}
f_0(\zeta)=\left\{\begin{aligned}
&\exp\left(\psi(\zeta)\cdot\log{\left(\frac{\zeta-1/2}{\zeta}\right)}\right)\ \text{if}\ r_1<|\zeta|<1,\\
&\frac{\zeta-1/2}{\zeta}\ \text{if}\ |\zeta|<r_1,
\end{aligned}\right.
\end{equation}
where ${\dis \log{\left(\frac{\zeta-1/2}{\zeta}\right)} }$ is a univalent branch of the $\log$,
which is well defined for $\zeta$ satisfying $|\zeta|\in(r_1,1)$.
From the definition of the function$f_0$ it follows that it can be extended to $\widetilde{V}\setminus U$
as $f_0(\zeta)=1$.\\
\indent
We consider the neighborhoods 
\begin{equation}\label{Neighborhoods}
\begin{aligned}
&U_1=\left\{\zeta\in U: |\zeta|<1/4\right\},\\
&U_2=\left\{\zeta\in U: |\zeta-1/2|<1/8\right\},\\
&U_2(\epsilon)=\left\{\zeta\in U: |\zeta-1/2|<1/8-\epsilon\right\}\Subset U_2,\\
&W=\widetilde{V}\setminus\left(\overline{U}_2(\epsilon)\cup\left\{|\zeta|\leq 1/8\right\}\right),
\end{aligned}
\end{equation}
satisfying the following equalities
\begin{equation}\label{SetEqualities}
\begin{aligned}
&U_1\cup W=\widetilde{V}\setminus \overline{U}_2(\epsilon),\\
&U_1\cap W=\left\{\zeta\in U: 1/8<|\zeta|<1/4\right\}.
\end{aligned}
\end{equation}
To construct a meromorphic function on $\widetilde{V}\setminus \overline{U}_2(\epsilon)$
satisfying \eqref{LocalEquality} and \eqref{IndicesEquality} we consider the smooth differential form $\alpha^{(0,1)}$
\begin{equation*}
\begin{aligned}
&\alpha^{(0,1)}=\frac{\bar\partial f_0}{f_0}\ \text{on}\ W,\\
&\alpha^{(0,1)}=0\ \text{on}\ U_1.
\end{aligned}
\end{equation*}
Using the solvability of the $\bar\partial$-equation on the open Riemann surface
$\widetilde{V}\setminus \overline{U}_2(\epsilon)$ (see \cite{Fo}) we obtain a smooth function $g$
such that $\bar\partial g=\alpha^{(0,1)}$ and the meromorphic function
$f(z)=e^{-g}\cdot f_0$ on $\widetilde{V}\setminus \overline{U}_2(\epsilon)$
having the only pole at $z^{(1)}$ and no zeros, and satisfying on
$\widetilde{V}\setminus\left(\overline{U}_2(\epsilon)\cup z^{(1)}\right)$ the condition
\begin{equation*}
\bar\partial f=-e^{-g} f_0\bar\partial g+e^{-g}\bar\partial f_0
=e^{-g}\cdot f_0\left(-\bar\partial g+\frac{\bar\partial f_0}{f_0}\right)=0.
\end{equation*}
\indent
From \eqref{fZero} we obtain that for $\zeta\in U,\ |\zeta|<r_1$ we have
\begin{equation*}
\log|f(\zeta)|^2=\log{|e^{-g(\zeta)}|^2}+\log{|\zeta-1/2|^2}-\log{|\zeta|^2},
\end{equation*}
and therefore
\begin{multline}\label{UoneEquality}
\frac{1}{2\pi i}\int_{bU_1}\partial\log{|f(\zeta)|^2}
=\frac{1}{2\pi i}\left(\int_{bU_1}\partial\log{|e^{-g(\zeta)}|^2}
+\int_{bU_1}\partial\log{|\zeta-1/2|^2}-\int_{bU_1}\partial\log{|\zeta|^2}\right)\\
=\frac{1}{2\pi i}\left(\int_{bU_1}\frac{-e^{-{\bar g}(\zeta)}e^{-g(\zeta)}}{|e^{-g(\zeta)}|^2}\partial g(\zeta)
+\int_{bU_1}\frac{{\bar\zeta}-1/2}{|\zeta-1/2|^2}d\zeta
-\int_{bU_1}\frac{\bar\zeta}{|\zeta|^2}d\zeta\right)\\
=\frac{1}{2\pi i}\left(-\int_{bU_1}\partial g(\zeta)+\int_{bU_1}\frac{d\zeta}{(\zeta-1/2)}
-\int_{bU_1}\frac{d\zeta}{\zeta}\right).
\end{multline}
Since $g$ is holomorphic in $\left\{\zeta\in U: |\zeta|<r_1\right\}\setminus \overline{U}_2(\epsilon)$,
we have $\partial g=d g$, and therefore
$$\int_{bU_1}\partial g(\zeta)=\int_{bU_1}dg(\zeta)=0.$$
Similarly, we have
$$\int_{bU_1}\frac{d\zeta}{(\zeta-1/2)}=0\ \text{and}\ \int_{bU_1}\frac{d\zeta}{\zeta}=2\pi i.$$
\indent
Then using the equalities above and similar equalities for ${\dis \int_{bU_2} }$ we obtain
equalities \eqref{IndicesEquality}.
\end{proof}

\begin{lemma}\label{TwoPointFunction}
Let $U_0\Subset \widetilde{V}$ be an open neighborhood in $\widetilde{V}$,
and let $z^{(1)}, z^{(2)} \in U\Subset \widetilde{V}\setminus\overline{U}_0$, where $U$ is a coordinate
neighborhood in $\widetilde{V}$ with coordinate function $\zeta$ such that $\zeta(z^{(1)})=0, \zeta(z^{(2)})=1/2$.
Then there exist neighborhoods $U_j\ni z^{(j)}\ (j=1,2)$ and a meromorhic function
$f$ on $\widetilde{V}\setminus\left(\overline{U}_0\cup z^{(1)}\cup z^{(2)}\right)$ such that
\begin{equation}\label{TwoPointLocalEquality}
\left\{\begin{aligned}
&\frac{1}{2\pi i}\int_{bU_1}\partial\log{|f(\zeta)|^2}=-1,\\
&\frac{1}{2\pi i}\int_{bU_2}\partial\log{|f(\zeta)|^2}=1.
\end{aligned}\right.
\end{equation}
\end{lemma}
\begin{proof}
As in the proof of Lemma~\ref{FirstStep} we consider a smooth function $\psi$
satisfying conditions \eqref{PsiFunction} and define the function $f_0$
on $U\setminus\left(z^{(1)}\cup z^{(2)}\right)$ by \eqref{fZero}.
Again from the definition of the function $f_0$ it follows that it can be extended to $\widetilde{V}\setminus U$
as $f_0(\zeta)=1$.\\
\indent
We consider the neighborhoods
\begin{equation}\label{TwoNeighborhoods}
\begin{aligned}
&U_1=\left\{\zeta\in U: |\zeta|<1/4\right\},\\
&U_2=\left\{\zeta\in U: |\zeta-1/2|<1/4\right\},\\
&W=\widetilde{V}\setminus\left(\overline{U}_0\cup\left\{|\zeta|\leq1/8\right\}
\cup\left\{|\zeta-1/2|\leq1/8\right\}\right),
\end{aligned}
\end{equation}
satisfying the following equalities
\begin{equation}\label{TwoSetEqualities}
\begin{aligned}
&\left(U_1\cup U_2\right)\cup W=\widetilde{V}\setminus \overline{U}_0,\\
&\left(U_1\cup U_2\right)\cap W=\left\{\zeta\in U: 1/8<|\zeta|<1/4, 1/8<|\zeta-1/2|<1/4\right\}.
\end{aligned}
\end{equation}
To construct a meromorphic function on $\widetilde{V}\setminus \overline{U}_0$
satisfying \eqref{TwoPointLocalEquality}
we consider the smooth differential form $\alpha^{(0,1)}$
\begin{equation*}
\begin{aligned}
&\alpha^{(0,1)}=\frac{\bar\partial f_0}{f_0}\ \text{on}\ W,\\
&\alpha^{(0,1)}=0\ \text{on}\ \left(U_1\cup U_2\right).
\end{aligned}
\end{equation*}
Then on the open Riemann surface $\widetilde{V}\setminus \overline{U}_0$
we consider a smooth function $g$ such that
$\bar\partial g=\alpha^{(0,1)}$ and the meromorphic function $f(z)=e^{-g}\cdot f_0$
on $\widetilde{V}\setminus \left(\overline{U}_0\cup z^{(1)}\cup z^{(2)}\right)$ satisfying
\begin{equation*}
\bar\partial f=-e^{-g} f_0\bar\partial g+e^{-g}\bar\partial f_0
=e^{-g}\cdot f_0\left(-\bar\partial g+\frac{\bar\partial f_0}{f_0}\right)=0.
\end{equation*}
\indent
From \eqref{fZero} we obtain that for $\zeta\in U,\ |\zeta|<r_1$ we have
\begin{equation*}
\log|f(\zeta)|^2=\log{|e^{-g(\zeta)}|^2}+\log{|\zeta-1/2|^2}-\log{|\zeta|^2},
\end{equation*}
and therefore, as in \eqref{UoneEquality}
\begin{equation*}
\frac{1}{2\pi i}\int_{bU_1}\partial\log{|f(\zeta)|^2}
=\frac{1}{2\pi i}\left(-\int_{bU_1}\partial g(\zeta)+\int_{bU_1}\frac{d\zeta}{(\zeta-1/2)}
-\int_{bU_1}\frac{d\zeta}{\zeta}\right).
\end{equation*}
Since $g$ is holomorphic in $\left\{\zeta\in U: |\zeta|<r_1\right\}\setminus\left(U_1\cup U_2\right)$,
we have $\partial g=d g$, and therefore
\begin{equation*}
\left\{\begin{aligned}
\int_{bU_1}\partial g(\zeta)=\int_{bU_1}dg(\zeta)=0,\\
\int_{bU_2}\partial g(\zeta)=\int_{bU_2}dg(\zeta)=0.
\end{aligned}\right.
\end{equation*}
Using the equalities above together with equalities
\begin{equation*}
\int_{bU_1}\frac{d\zeta}{(\zeta-1/2)}=\int_{bU_2}\frac{d\zeta}{\zeta}=0
\end{equation*}
we obtain equalities \eqref{TwoPointLocalEquality}.
\end{proof}
\indent
The following statement formulated in the terminology of the book \cite{Fo}
is a corollary of Lemma~\ref{TwoPointFunction}:
\begin{corollary}
Any divisor $D$ with $\deg D=0$ on an open Riemann surface is solvable.
\end{corollary}

\begin{lemma}\label{Induction}
Let $z^{(1)}, z^{(2)}$ be two points in $\widetilde{V}$, let $U_j\ni z^{(j)}$
be two neighborhoods of those points in $\widetilde{V}$, and let $U_j(\epsilon)\subset U_j$ be slightly
smaller neighborhoods. Then there exists a holomorphic function $f$ on
$\widetilde{V}\setminus\left(\overline{U}_1(\epsilon)\cup \overline{U}_2(\epsilon)\right)$ such that
\begin{equation}\label{TwoPointEquality}
\left\{\begin{aligned}
&\frac{1}{2\pi i}\int_{bU_1}\partial\log{|f(\zeta)|^2}=-1,\\
&\frac{1}{2\pi i}\int_{bU_2}\partial\log{|f(\zeta)|^2}=1.
\end{aligned}\right.
\end{equation}
\end{lemma}
\begin{proof}
For the points $z^{(1)}, z^{(2)}$ we consider a sequence of points
$z^{(1)}=w^{(1)},\dots,w^{(r)}=z^{(2)}$ such that every two consecutive points
in the sequence belong to the same coordinate neighborhood. Applying
Lemma~\ref{FirstStep} we construct two meromorphic functions: $g_1$ on
$\widetilde{V}\setminus \overline{U}_1(\epsilon)$ with zero at $w^{(2)}$ and $g_{r-1}$
on $\widetilde{V}\setminus \overline{U}_2(\epsilon)$ with pole at $w^{(r-1)}$. Then using
Lemma~\ref{TwoPointFunction} we construct a sequence of meromorphic functions
$g_j (j=2,\dots,r-2)$ on $\widetilde{V}\setminus
\left(\overline{U}_1(\epsilon)\cup \overline{U}_2(\epsilon)\right)$
such that $g_j$ has a pole at $w^{(j)}$ and a zero at $w^{(j+1)}$. Defining then
\begin{equation*}
f=\prod_{j=1}^{r-1}g_j
\end{equation*}
we obtain the sought function.
\end{proof}

Proof of Proposition~\eqref{ZeroPeriods-1}
\begin{proof}
Using Lemma~\ref{Induction} we construct a set of holomorphic functions
$\left\{f_r\ (r=1,\dots,m-1)\right\}$ satisfying
\eqref{TwoPointEquality} on $\widetilde{V}\setminus\left(V_r\cup V_{r+1}\right)$,
and define functions $h_r=\log{|f_r|^2}$.
Then we prove the existence for an arbitrary harmonic $u\in C^1(\overline{V})$ of the coefficients
satisfying \eqref{aCoefficients}.
Let $l\in\Z$ be the number such that there exist coefficients $a_j$ for $j>l$ satisfying
\begin{equation}\label{lCondition}
\int_{bV_j}\partial\left(u-\sum_{l+1}^{m-1}a_j h_j\right)=0\ \text{for}\ j=l+2,\dots,m.
\end{equation}
We prove the proposition by induction with respect to $l$. Namely, assuming that for
some $l$ equality \eqref{lCondition} is satisfied for $j>l+1$, we define
$a_l=\int_{bV_{l+1}}\partial\left(u-\sum_{l+1}^{m-1}a_j h_j\right)$ and obtain that
equality
$$\int_{bV_j}\partial\left(u-\sum_{l}^{m-1}a_j h_j\right)=0$$
is satisfied for $j\geq l+1$.\end{proof}

\indent
In the following proposition we prove the second set of equalities in \eqref{ZeroIntegrals} for the function
$u-\sum_{j=1}^{m-1}a_jh_j$ constructed in Proposition~\ref{ZeroPeriods-1}.

\begin{proposition}\label{ZeroPeriods-2}
Let $u$ be a harmonic function on $V$, and let ${\dis h=\sum_{j=1}^{m-1}a_jh_j }$ be the function defined in
Proposition~\ref{ZeroPeriods-1}. Then the function $u-h$ satisfies the second set of equalities in
\eqref{ZeroIntegrals}.
\end{proposition}
\begin{proof}
We consider a basis $\left\{\omega_j\right\}_{j=1}^p$ of holomorphic forms in $H^1(\widetilde{V},\cal{O})$
(see \cite{S}), and the corresponding $(2p\times 2p)$ period matrix
\begin{equation*}
\P=\left[\begin{array}{ccc}
\alpha_{1}&\cdots&\alpha_{2p}\\
\overline{\alpha}_{1}&\cdots&\overline{\alpha}_{2p}
\end{array}\right],
\end{equation*}
where
\begin{equation*}
\alpha_i=\left(\begin{array}{c}
\int_{\gamma_i}\omega_1\\
\vdots\\
\int_{\gamma_{i}}\omega_p
\end{array}\right),\hspace{0.2in}
\overline{\alpha}_i=\left(\begin{array}{c}
\int_{\gamma_i}\overline{\omega}_1\\
\vdots\\
\int_{\gamma_i}\overline{\omega}_p
\end{array}\right)\hspace{0.1in}\text{for}\ i=1,\dots,2p.
\end{equation*}
\indent
Normalizing the forms $\left\{\omega_j\right\}_{j=1}^p$ we transform the matrix $\P$ into
\begin{equation}\label{P-matrix}
\P=\left[\begin{array}{cc}
I&A+iB\\
I&A-iB
\end{array}\right]
=\left[\begin{array}{cccccc}
1&\cdots&0&\alpha_{p+1,1}&\cdots&\alpha_{2p,1}\\
\vdots&\ddots&\vdots&\vdots&\ddots&\vdots\\
0&\cdots&1&\alpha_{p+1,p}&\cdots&\alpha_{2p,p}\\
1&\cdots&0&\bar{\alpha}_{p+1,1}&\cdots&\bar{\alpha}_{2p,1}\\
\vdots&\ddots&\vdots&\vdots&\ddots&\vdots\\
0&\cdots&1&\bar{\alpha}_{p+1,p}&\cdots&\bar{\alpha}_{2p,p}
\end{array}\right]
\end{equation}
with symmetric matrix $A+iB$, and positive definite $B$ (\cite{S}).\\
\indent
For the $2p$-vector $c_i=\int_{\gamma_i}\partial(u-h)$ we consider the system
of linear equations
\begin{equation}\label{PeriodVectorEquations}
\left[\begin{array}{ccc}
c_1&\cdots&c_{2p}
\end{array}\right]
=\left[\begin{array}{ccc}
\zeta_1&\cdots&\zeta_{2p}
\end{array}\right]
\left[\begin{array}{cc}
I&A+iB\\
I&A-iB
\end{array}\right]
\end{equation}
with solution
$\left[\begin{array}{ccc}
\lambda_1&\lambda_2
\end{array}\right] \in \C^{2p}$, where $\lambda_1, \lambda_2 \in \C^p$ are defined by the formula
\begin{equation}\label{lambdaSolution}
\left[\begin{array}{ccc}
\lambda_1&\lambda_2
\end{array}\right]
=\left[\begin{array}{ccc}
c_1&\cdots&c_{2p}
\end{array}\right]\left[\begin{array}{cc}
(iAB^{-1}+I)/2&(-iAB^{-1}+I)/2\\
-iB^{-1}/2&iB^{-1}/2
\end{array}\right].
\end{equation}
We denote for $\lambda\in \C^p$
\begin{equation*}
\left\langle \lambda, \omega\right\rangle=\sum_{j=1}^p\lambda^{(j)}\omega_j,
\end{equation*}
and for
$\left[\begin{array}{ccc}
\lambda_1&\lambda_2
\end{array}\right]$
defined in \eqref{lambdaSolution} obtain from \eqref{ZeroIntegrals} the following equalities
\begin{equation}\label{ZeroNewIntegrals}
\begin{aligned}
&\int_{\gamma_i}\big(\partial(u-h)-\langle\lambda_1, \omega\rangle
-\langle\lambda_2, \overline{\omega}\rangle\big)=0
\hspace{0.1in}\text{for}\ i=1,\dots,2p,\\
&\int_{\sigma_r}\big(\partial(u-h)-\langle\lambda_1, \omega\rangle
-\langle\lambda_2, \overline{\omega}\rangle\big)=0
\hspace{0.1in}\text{for}\ r=1,\dots,m.
\end{aligned}
\end{equation}
Equalities \eqref{ZeroNewIntegrals} imply the existence of a harmonic function $g$ on $V$ such that
\begin{equation*}
\partial(u-h)-\langle\lambda_1, \omega\rangle
-\langle\lambda_2, \overline{\omega}\rangle=dg,
\end{equation*}
or equivalently,
\begin{equation}\label{VectorpsiConditions}
\left\{
\begin{aligned}
&\partial\left(u-h-g\right)=\langle\lambda_1, \omega\rangle,\\
&\bar\partial g=-\langle\lambda_2, \overline{\omega}\rangle.
\end{aligned}
\right.
\end{equation}
\indent
We notice that second set of equalities in \eqref{ZeroNewIntegrals} is satisfied automatically since $h$
satisfies Lemma~\ref{ZeroPeriods-1} and the forms
$\langle\lambda_1, \omega\rangle$ and $\langle\lambda_2, \overline{\omega}\rangle$ are closed
in $\widetilde{V}$ and therefore in each $V_r$.
Also function $g$ is harmonic since
\begin{equation*}
\partial\bar\partial g=-\langle\lambda_2, \partial\overline{\omega}\rangle=0,
\end{equation*}
because the forms $\left\{\omega_j\right\}_{j=1}^p$ are holomorphic, and the forms
$\left\{\overline{\omega}_j\right\}_{j=1}^p$ are antiholomorphic.\\
\indent
To simplify system \eqref{VectorpsiConditions}
we rewrite the first equality in \eqref{ZeroNewIntegrals}
for $\left[\begin{array}{ccc}
\lambda_1&\lambda_2
\end{array}\right]$ as
\begin{equation*}
\begin{aligned}
&c_k=\lambda_1^{(k)}+\lambda_2^{(k)}\hspace{0.1in}\text{for}\ k=1,\dots,p,\\
&c_k=\left\langle \lambda_1, \left[A+iB\right]_k\right\rangle
+\left\langle \lambda_2, \left[A-iB\right]_k\right\rangle\hspace{0.1in}
\text{for}\ k=p+1,\dots,2p
\end{aligned}
\end{equation*}
and obtain equalities
\begin{equation*}
\begin{aligned}
&\int_{\gamma_k}\partial(u-h)=c_k=\lambda_1^{(k)}+\lambda_2^{(k)}
=\lambda_1^{(k)}\int_{\gamma_k}\omega_k
+\lambda_2^{(k)}\int_{\gamma_k}\overline{\omega}_k
\hspace{0.1in}\text{for}\ k=1,\dots,p,\\
&\int_{\gamma_k}\partial(u-h)=c_k=\sum_{j=1}^p\lambda_1^{(j)}\alpha_{jk}
+\sum_{j=1}^p\lambda_2^{(j)}\overline{\alpha}_{jk}\\
&\hspace{1.0in}=\sum_{j=1}^p\lambda_1^{(j)}\int_{\gamma_k}\omega_j
+\sum_{j=1}^p\lambda_2^{(j)}\int_{\gamma_k}\overline{\omega}_j
\hspace{0.1in}\text{for}\ k=p+1,\dots,2p,
\end{aligned}
\end{equation*}
where we denoted
\begin{equation*}
\alpha_{jk}=a_{jk}+ib_{jk}=\int_{\gamma_k}\omega_j,\hspace{0.1in}
\text{an element of}\hspace{0.1in}
\alpha_k=\left(\begin{array}{c}
\int_{\gamma_k}\omega_1\\
\vdots\\
\int_{\gamma_k}\omega_p
\end{array}\right).
\end{equation*}
\indent
Since function $(u-h)$ is real-valued, we have equality
\begin{equation*}
\partial(u-h)=\frac{1}{2}\Big[d(u-h)+id^c(u-h)\Big],
\end{equation*}
where $d(u-h)$ and $d^c(u-h)$ are real valued forms.
Then using equalities
\begin{equation*}
\int_{\gamma_i}d(u-h)=0\hspace{0.1in}\text{for}\ i=1,\dots,2p,
\end{equation*}
we obtain that numbers $c_1,\cdots,c_{2p}$ are imaginary, and therefore
\begin{equation*}
\begin{aligned}
&\lambda_1=\tau+i\theta_1,\\
&\lambda_2=-\tau+i\theta_2,
\end{aligned}
\end{equation*}
with $\tau, \theta \in \R^p$.
Using equalities
\begin{equation*}
\text{Re}\left[(\tau+i\theta_1)(A+iB)+(-\tau+i\theta_2)(A-iB)\right]
=-\theta_1B+\tau A-\tau A+\theta_2B=0,
\end{equation*}
and the nondegeneracy of $B$, we obtain that $\theta_1=\theta_2$,
and therefore $\lambda_2=-\overline{\lambda_1}$.\\
\indent
Denoting $\lambda=\lambda_1$,  we rewrite system \eqref{VectorpsiConditions} as
\begin{equation}\label{d-dBar-Conditions}
\left\{
\begin{aligned}
&\partial\left(u-h-g\right)=\langle\lambda, \omega\rangle,\\
&\bar\partial g=\langle\bar{\lambda}, \overline{\omega}\rangle.
\end{aligned}
\right.
\end{equation}
\indent
In the following two lemmas we compute the constants $\lambda$ in the right-hand sides
of equalities \eqref{d-dBar-Conditions} separately in cases $p=1$ and $p>1$.

\begin{lemma}\label{ZeroConstants-1}
Let $\widetilde{V}$ be a torus, i.e. $p=1$, and let $u$, $h$, and $g$ be as in Proposition~\ref{ZeroPeriods-2}. Then the constant $\lambda$ in the right-hand sides of equalities \eqref{d-dBar-Conditions} is zero.
\end{lemma}
\begin{proof}
Since $\widetilde{V}$ is a torus, we can take $\omega=d\zeta$, where $\zeta=\xi+i\eta$
is the coordinate in $\C$ - the universal covering of $\widetilde{V}$ - and rewrite the second
equality in \eqref{d-dBar-Conditions} as
\begin{equation}\label{lambdaFormEquality}
\left(g_{\xi}+ig_{\eta}\right)d\bar\zeta=2\bar\lambda d\bar\zeta.
\end{equation}
\indent
Assuming that the fundamental region of $\widetilde{V}$ is the parallelogram
\begin{equation*}
{\cal F}=\left\{(u,v)\in \R^2=\C, 0<u<L, 0<v<M\right\},
\end{equation*}
where $u=a_{11}\xi+a_{12}\eta,\ v=a_{21}\xi+a_{22}\eta$ with some real-valued nondegenerate matrix
${\dis \left[\begin{array}{cc}
a_{11}&a_{12}\\
a_{21}&a_{22}
\end{array}\right] }$,
we obtain that the partial derivatives with respect to $\xi, \eta$ satisfy equalities
\begin{equation*}
\begin{aligned}
&g_{\xi}=a_{11}g_u+a_{21}g_v,\\
&g_{\eta}=a_{12}g_u+a_{22}g_v.
\end{aligned}
\end{equation*}
Then equality \eqref{lambdaFormEquality} can be rewritten as
\begin{equation}\label{lambdaEquality}
\left[\left(a_{11}g_u+a_{21}g_v\right)+i\left(a_{12}g_u+a_{22}g_v\right)\right]
=\left[(a_{11}+ia_{12})g_u+(a_{21}+ia_{22})g_v\right]
=2\bar\lambda.
\end{equation}
\indent
Considering the Fourier series of $g_u$ with respect to the variable $u$ in the region
\begin{equation*}
{\cal R}_{u}(\epsilon)={\cal F}\cap\left\{|v|<\epsilon\right\},
\end{equation*}
where $\epsilon$ is a sufficiently small number, we obtain the series
\begin{equation}\label{g-uParallelogramSeries}
g_u(u,v)=\sum_{k\neq 0}\alpha_k(v)e^{2\pi ik(u/L)},
\end{equation}
where the coefficients $\alpha_k(v)\ (k\neq 0)$ are computed by the formula
\begin{equation*}
\alpha_k(v)=\frac{1}{L}\int_0^Lg_u(t,v)e^{-2\pi ik(t/L)}dt.
\end{equation*}
We notice that the zeroth order term in the series \eqref{g-uParallelogramSeries} is absent,
because the function $g$ takes the same value at the end points of each interval $[0, L]\times v$,
since these points are identified on the Riemann surface $\widetilde{V}$.\\
\indent
Similarly, in the region ${\cal R}_{v}(\epsilon)={\cal F}\cap\left\{|u|<\epsilon\right\}$
we obtain the series
\begin{equation}\label{g-vParallelogramSeries}
g_v(u,v)=\sum_{k\neq 0}\beta_k(u)e^{2\pi ik(v/M)},
\end{equation}
where the coefficients $\beta_k(u)\ (k\neq 0)$ are computed by the formula
\begin{equation*}
\beta_k(u)=\frac{1}{M}\int_0^Mg_v(u,t)e^{-2\pi ik(t/M)}dt.
\end{equation*}
\indent
Substituting the series \eqref{g-uParallelogramSeries} and \eqref{g-vParallelogramSeries} in equality
\eqref{lambdaEquality} in the region
\begin{equation*}
{\cal R}_{uv}\left(L/N,M/N\right)={\cal F}\cap\left\{0<u<L/N, 0<v<M/N\right\},
\end{equation*}
where $N\in\Z$ is large enough, we obtain equality 
\begin{equation}\label{Fourier-Equality}
(a_{11}+ia_{12})\left(\sum_{k\neq 0}\alpha_k(v)e^{2\pi ik(u/L)}\right)
+(a_{21}+ia_{22})\left(\sum_{k\neq 0}\beta_k(u)e^{2\pi ik(v/M)}\right)=2\bar\lambda.
\end{equation}
Then, considering the Fourier series of coefficients $\alpha_k$ and $\beta_k$
with respect to $v$ and $u$ respectively
\begin{equation*}
\alpha_k(v)=\sum_{j\in\Z}\alpha_{kj}e^{2\pi ij(Nv/M)},\hspace{0.2in}
\beta_k(u)=\sum_{j\in\Z}\beta_{kj}e^{2\pi ij(Nu/L)},
\end{equation*}
and comparing the double Fourier series in the right and
left-hand sides of \eqref{Fourier-Equality}, we obtain the following equality
in ${\cal R}_{u,v}(L/N,M/N)$
\begin{equation*}
(a_{11}+ia_{12})\left(\sum_{k\neq 0,\ j\in\Z}\alpha_{kj}e^{2\pi i\left[jNv/M+ku/L\right]}\right)
+(a_{21}+ia_{22})\left(\sum_{k\neq 0,\ j\in\Z}\beta_{kj}e^{2\pi i\left[kv/M+jNu/L\right]}\right)
=2\bar\lambda,
\end{equation*}
which cannot be satisfied unless $\lambda=0$.
\end{proof}
\indent
In the lemma below we prove the statement similar to Lemma~\ref{ZeroConstants-1} for the case
$p>1$ with the Riemann surface $\widetilde{V}$ having the unit disk $\D$ as the universal covering.

\begin{lemma}\label{ZeroConstants-2}
Let $\widetilde{V}$ be a Riemann surface of genus $p>1$, and let $u$, $h$, and $g$
be as in Proposition~\ref{ZeroPeriods-2}. Then the constants $\left\{\lambda^{(j)}\right\}_{j=1}^p$
in the right-hand sides of equalities \eqref{d-dBar-Conditions} are zeros.
\end{lemma}
\begin{proof}
As a fundamental region ${\cal F}\subset\D$ corresponding to the compact Riemann surface
$\widetilde{V}$ we choose a polygon with vertices $\left\{P_s\right\}_{s=1}^{4p}$
and hyperbolic geodesic sides $\widetilde{[P_s, P_{s+1}]}$ (see for example \cite{S, Sp, J}).
We also consider the \lq\lq Euclidean\rq\rq\ polygon ${\cal P}\subset \D$
such that ${\cal P}\supset {\cal F}$, which is constructed on the same vertices
$\left\{P_s\right\}_{s=1}^{4p}$ with the sides $[P_s, P_{s+1}]$, and denote
\begin{equation*}
{\cal P}\setminus{\cal F}=\bigcup_{s=1}^{4p}{\cal L}_s,
\end{equation*}
where $4p$ mutually disjoint regions ${\cal L}_s$ are bounded by the arcs of geodesic sides of the
fundamental polygon ${\cal F}$ and straight linear sides of the polygon ${\cal P}$.\\
\indent
Throughout the proof of the Lemma we assume the functions considered below to be defined on a sufficiently
large neighborhood of ${\cal F}$ in $\D$, containing ${\cal P}$, via the automorphy condition
\begin{equation*}
f(Tz)=f(z),
\end{equation*}
where $T$ is an element of the Fuchsian group corresponding to the Riemann surface $\widetilde{V}$.\\
\indent
Using the standard identification scheme of the sides of ${\cal F}$ (see \cite{S}) we divide
the set of sides of ${\cal F}$ into $p$ blocks of the form
$\left\{\varsigma_j\tau_j\varsigma_j^{-1}\tau_j^{-1}\right\}_{j=1}^p$, and assume that the holomorphic
differential forms $\left\{\omega_j\right\}_{j=1}^p$ from the basis in $H^1(\widetilde{V},\cal{O})$
(see Proposition~\ref{ZeroPeriods-2}) are chosen so that
\begin{equation}\label{BasisConditions}
\left\{
\begin{aligned}
&\int_{\tau_k}\omega_j=\delta_{kj},\\
&\int_{\varsigma_k}\omega_j=\alpha_{p+k,j},
\end{aligned}\right.
\end{equation}
where $\delta_{kj}$ is the Kronecker's delta, and $\alpha_{p+k,j}$ is the element of matrix $\P$
in \eqref{P-matrix}.
Using the introduced notations we can rewrite the second equality in \eqref{d-dBar-Conditions} as
\begin{equation}\label{lambdaSum}
\left(g_{\xi}+ig_{\eta}\right)d\bar\zeta=2\sum_{j=1}^p\bar\lambda_j\bar\omega_j,
\end{equation}
where $\zeta=\xi+i\eta$ is a holomorphic coordinate in the unit disk $\D$.\\
\indent
In what follows we fix five consecutive vertices $P_s (s=1,\dots,5)$ of ${\cal F}$ and the corresponding
sides $\varsigma_1=\widetilde{[P_1, P_2]}$, $\tau_1=\widetilde{[P_2, P_3]}$,
$\varsigma_1^{-1}=\widetilde{[P_3, P_4]}$, $\tau_1^{-1}=\widetilde{[P_4, P_5]}$,
and consider the coordinate system with the origin at the vertex $P_3$ and axes $u$ and $v$
being the sides of $\cal P$ - $[P_3, P_2]$ and $[P_3, P_4]$ respectively.
Since the coordinates $(u,v)$ satisfy linear relations
$u=a_{11}\xi+a_{12}\eta+u_0,\ v=a_{21}\xi+a_{22}\eta+v_0$
with some real-valued nondegenerate matrix
\begin{equation}\label{matrixA}
A=\left[\begin{array}{cc}
a_{11}&a_{12}\\
a_{21}&a_{22}
\end{array}\right],
\end{equation}
the partial derivatives with respect to $\xi, \eta$ satisfy equalities
\begin{equation*}
\begin{aligned}
&g_{\xi}=a_{11}g_u+a_{21}g_v,\\
&g_{\eta}=a_{12}g_u+a_{22}g_v,
\end{aligned}
\end{equation*}
and, therefore equality \eqref{lambdaSum} can be rewritten as
\begin{equation}\label{Equality-in-uv}
\left[\left(a_{11}g_u+a_{21}g_v\right)+i\left(a_{12}g_u+a_{22}g_v\right)\right]d\bar\zeta
=\left[(a_{11}+ia_{12})g_u+(a_{21}+ia_{22})g_v\right]d\bar\zeta
=2\sum_{j=1}^p\bar\lambda_j\bar\omega_j.
\end{equation}
\indent
Our goal is to prove by comparing the Fourier coefficients of the right and left-hand sides in equality
\eqref{Equality-in-uv} that it cannot hold unless $\left\{\lambda^{(j)}=0\right\}_{j=1}^p$.
We consider the parallelogram on the vertices $P_2, P_3, P_4$ with vertex $P_2$ having coordinates
$(u,v)=(L,0)$, and vertex $P_4$ having coordinates $(u,v)=(0,M)$.
Considering the Fourier series of $g_u$ with respect to variable $u$ in the region
\begin{equation*}
{\cal R}_{\tau_1}(\epsilon)={\cal P}
\cap\left\{(u,v):0<v<\epsilon, 0\leq u\leq L\right\},
\end{equation*}
where $\epsilon$ is a sufficiently small number, we obtain the series
\begin{equation}\label{g-uSeries}
g_u(u,v)=\sum_{k\in\Z}\alpha_k(v)e^{2\pi ik(u/L)}.
\end{equation}
The coefficients $\alpha_k(v)$ of the series above are computed by the formula
\begin{equation*}
\alpha_k(v)=\frac{1}{L}\int_0^L g_u(t,v)e^{-2\pi ik(t/L)}dt,
\end{equation*}
and the zeroth order term satisfies condition
\begin{equation}\label{u-Condition}
\left|\alpha_0(v)\right|<C\cdot v\hspace{0.1in}\text{as}\
v\to 0\hspace{0.1in}\text{with some constant}\ C>0,
\end{equation}
because function $g(u,0)$ takes the same value at the end points of the interval
$[P_3, P_2]$, since these points are identified on the Riemann surface $\widetilde{V}$.\\
\indent
Similarly, we construct the series
\begin{equation}\label{g-vSeries}
g_v(u,v)=\sum_{k\in\Z}\beta_k(u)e^{2\pi ik(v/M)}
\end{equation}
in the region
\begin{equation*}
{\cal R}_{\varsigma_1^{-1}}(\delta)
={\cal P}\cap\left\{(u,v):0<u<\delta, 0\leq v\leq M\right\},
\end{equation*}
where $M$ is the length of ${[P_3, P_4]}$, and the zeroth order term satisfies condition
\begin{equation}\label{v-Condition}
\left|\beta_0(u)\right|<C\cdot u\hspace{0.1in}\text{as}\ u\to 0.
\end{equation}
\indent
We represent $\bar\omega_j=f_j(\zeta)d\bar\zeta$, rewrite the form in the right-hand side of
\eqref{Equality-in-uv} as
\begin{equation*}
2\sum_{j=1}^p\bar\lambda_j\bar\omega_j
=2\left(\sum_{j=1}^p\bar\lambda_jf_j(\zeta)\right)d\bar\zeta,
\end{equation*}
and consider the Fourier series of the functions $\left\{f_j(\zeta)\right\}_{j=1}^p$
in the  region ${\cal R}_{\tau_1}(\epsilon)$. For $j>1$ we use equality
\begin{equation*}
d\bar\zeta=d(\xi-i\eta)=(b_{11}du+b_{12}dv)-i(b_{21}du+b_{22}dv)
=(b_{11}-ib_{21})du+(b_{12}-ib_{22})dv
\end{equation*}
with matrix ${\dis B=\left[\begin{array}{cc}
b_{11}&b_{12}\\
b_{21}&b_{22}
\end{array}\right] }$ being the inverse of matrix $A$ from \eqref{matrixA} to obtain equality
\begin{equation}\label{ZeroFourier}
\int_0^L f_j(u,0)du=\frac{1}{b_{11}-ib_{21}}\int_0^L
f_j(\zeta)d\bar\zeta=\frac{1}{b_{11}-ib_{21}}\int_{\tau_{1}}f_j(\zeta)d\bar\zeta=0,
\end{equation}
where in the last equality we used the closedness of the form $\omega_j$ and
the first equality from \eqref{BasisConditions}.\\
\indent
Then for the function $F(\zeta)=\sum_{j=2}^p\bar\lambda_jf_j(\zeta)$
in the region ${\cal R}_{\tau_1}(\epsilon)$ we obtain the representation
\begin{equation}\label{F_uFourier}
F(\zeta)=\sum_{k\in\Z}\mu_k(v)e^{2\pi ik(u/L)}
\end{equation}
with
\begin{equation*}
\mu_k(v)=\frac{1}{L}\int_0^LF(u,v)e^{-2\pi ik(u/L)}du.
\end{equation*}
We notice that from equality \eqref{ZeroFourier} follows the estimate
\begin{equation}\label{F_uCondition}
\left|\mu_0(v)\right|<C\cdot v\hspace{0.1in}\text{as}\ v\to 0.
\end{equation}
\indent
For the function $f_1$, similarly to \eqref{ZeroFourier} we obtain equality
\begin{equation*}
\int_0^L f_1(u,0)du
=\frac{1}{b_{11}-ib_{21}}\int_0^L f_1(\zeta)d\bar\zeta
=\frac{1}{b_{11}-ib_{21}}\int_{\tau_{1}}f_1(\zeta)d\bar\zeta=\frac{1}{b_{11}-ib_{21}},
\end{equation*}
and, therefore, for the Fourier series
\begin{equation}\label{f_1Fourier}
f_1(\zeta)=\sum_{k\in\Z}\nu_k(v)e^{2\pi ik(u/L)}
\end{equation}
with
\begin{equation*}
\nu_k(v)=\frac{1}{L}\int_0^L f_1(u,v)e^{-2\pi ik(u/L)}du
\end{equation*}
we have the estimate
\begin{equation}\label{f_1Condition}
\left|\nu_0(v)-\frac{1}{L(b_{11}-ib_{21})}\right|<C\cdot v\hspace{0.1in}\text{as}\ v\to 0.
\end{equation}
\indent
Substituting series \eqref{g-uSeries}, \eqref{g-vSeries}, \eqref{F_uFourier}, and
\eqref{f_1Fourier} into equality \eqref{Equality-in-uv} we obtain in
\begin{equation*}
{\cal R}_{\tau_1\varsigma_1}(\epsilon,\delta)
={\cal R}_{\tau_1}(\epsilon)\cap{\cal R}_{\varsigma_1^{-1}}(\delta)
\end{equation*}
the following equality
\begin{multline}\label{SingleFourier}
(a_{11}+ia_{12})\left(\sum_{k\in\Z}\alpha_k(v)e^{2\pi ik(u/L)}\right)
+(a_{21}+ia_{22})\left(\sum_{k\in\Z}\beta_k(u)e^{2\pi ik(v/M)}\right)\\
=\sum_{k\in\Z}\mu_k(v)e^{2\pi ik(u/L)}
+\bar\lambda_1\sum_{k\in\Z}\nu_k(v)e^{2\pi ik(u/L)}.
\end{multline}
\indent
If we choose ${\dis \epsilon=\frac{M}{N}, \delta=\frac{L}{N} }$ with sufficiently large $N\in \Z$,
and consider Fourier series of the functions
\begin{equation*}
\alpha_k(v),\hspace{0.1in}\beta_k(u),\hspace{0.1in}\mu_k(v),\hspace{0.1in}
\text{and}\hspace{0.2in}\nu_k(v)
\end{equation*}
in respectively $v$, $u$, $v$, and $v$, then, from \eqref{SingleFourier} we obtain in
${\cal R}_{\tau_1\varsigma_1}(M/N,L/N)$ the following equality
\begin{multline}\label{DoubleFourier}
(a_{11}+ia_{12})\left(\sum_{j,k\in\Z}\alpha_{jk}e^{2\pi i\left[jNv/M+ku/L\right]}\right)
+(a_{21}+ia_{22})\left(\sum_{j,k\in\Z}\beta_{jk}e^{2\pi i\left[jNu/L+kv/M\right]}\right)\\
=\sum_{j,k\in\Z}\mu_{jk}e^{2\pi i\left[jNv/M+ku/L\right]}
+\bar\lambda_1\sum_{j,k\in\Z}\nu_{jk}e^{2\pi i\left[jNv/M+ku/L\right]}.
\end{multline}
\indent
Comparing the coefficients of the right and left-hand sides of \eqref{DoubleFourier} for $k=0$,
and using estimates \eqref{u-Condition}, \eqref{v-Condition}, \eqref{F_uCondition}, and \eqref{f_1Condition}
we obtain that unless $\lambda_1=0$, equality \eqref{DoubleFourier} cannot hold
in ${\cal R}_{\tau_1\varsigma_1}(M/N,L/N)$, since
\begin{equation*}
\begin{aligned}
&\left|\nu_{00}\left(N\right)\right|
=\left|\frac{N}{M}\int_0^{M/N}dv\frac{1}{L}\int_0^L f_1(t,v)dt\right|
\sim\left(\frac{1}{L(b_{11}-ib_{21})}+{\cal O}\left(\frac{1}{N}\right)\right),\\
&\left|\mu_{00}\left(N\right)\right|
=\left|\frac{N}{M}\int_0^{M/N}dv\frac{1}{L}\int_0^L F(t,v)dt\right|
\sim\frac{1}{L}\cdot{\cal O}\left(\frac{1}{N}\right),\\
&\left|\alpha_{00}\left(N\right)\right|
=\left|\frac{N}{M}\int_0^{M/N}dv\frac{1}{L}\int_0^L g_u(t,v)dt\right|
\sim\frac{1}{L}\cdot{\cal O}\left(\frac{1}{N}\right),\\
&\left|\beta_{00}\left(N\right)\right|
=\left|\frac{N}{L}\int_0^{L/N}du\frac{1}{M}\int_0^Mg_v(u,t)dt\right|
\sim\frac{1}{M}\cdot{\cal O}\left(\frac{1}{N}\right).
\end{aligned}
\end{equation*}
Since the block $\varsigma_1\tau_1\varsigma_1^{-1}\tau_1^{-1}$ and the corresponding form $\omega_1$
were chosen arbitrarily, we obtain equality $\lambda_j=0$ for $j=1,\dots,p$.
This completes the proof of Lemma~\ref{ZeroConstants-2}.
\end{proof}

\indent
To finish the proof of Proposition~\ref{ZeroPeriods-2} we use the results
of Lemmas~\ref{ZeroConstants-1} and \ref{ZeroConstants-2} in equality \eqref{PeriodVectorEquations},
and obtain that the second set of equalities in \eqref{ZeroIntegrals} is satisfied for
the function $u-h$ constructed in Proposition~\ref{ZeroPeriods-1}.
\end{proof}

This completes the proof of Proposition~\ref{Periods}.
\end{proof}

\section{Proof of Theorem~\ref{Main}.}\label{proof}

\indent
From equalities \eqref{ZeroIntegrals} we obtain that for any closed curve $\gamma$ in $V$
we have
\begin{equation*}
\int_{\gamma}\partial(u-h)=0,
\end{equation*}
and therefore by fixing a point $z^{*}\in V$ and defining for $z\in V$
\begin{equation}\label{HolomorphicFunction}
f(z)=2\int_{z^{*}}^z\partial (u-h)
\end{equation}
we obtain a holomorphic function $f$ on $V$ such that
\begin{equation}\label{HolomorphicConditions}
\begin{aligned}
&f=(u-h)+iv,\\
&df=2\ \partial(u-h).
\end{aligned}
\end{equation}
\indent
In our construction of the boundary representation formula for harmonic functions on $V$
we will use the formula from our earlier paper \cite{P} for boundary representation of holomorphic
functions on open Riemann surfaces as in \eqref{AffineCurve}. However,
the needed formula is proven in \cite{P} under additional assumptions that the holomorphic function
is defined not only on $V$, but on some neighborhood $U^{\epsilon}\subset \*S^{5}(1)$ as in \eqref{Uepsilon}, and has negative
homogeneity there. In the lemma below we eliminate those additional assumptions.

\begin{lemma}\label{Extension}
Let $V \subset\widetilde{V}\subset U^{\epsilon}\subset \*S^{5}(1)$ be as in \eqref{AffineCurve}
and \eqref{Uepsilon}, and such that $V\subset \left\{\C\P^2\setminus\left\{z_0=0\right\}\right\}$,
and let $f$ be a holomorphic function on $V$.\\
\indent
Then there exist an $\epsilon>0$ and a holomorphic function $g$ of homogeneity $(-1)$ on $U^{\epsilon}$
such that
\begin{equation}\label{gEquality}
z_0\cdot g\Big|_{V}=f.
\end{equation}
\end{lemma}
\begin{proof}
In the first step we construct an extension of $f$ to the function ${\tilde f}$ on an $\epsilon$-neighborhood
of $V$ in $\C^2=\C\P^2\setminus\left\{z_0=0\right\}$.
To construct this extension we consider the holomorphic
normal bundle $\cal{N}$ of $V$ in $\C\P^2$.
Using its triviality (see \cite{Fo}), we obtain the existence of a nonzero section
$n(z)=(n_1(z),n_2(z))\in \cal{T}_z(\C^2)=\C^2$, where we identify the normal subspace
$\cal{N}(z)$ with the factor-space $\cal{T}_z(\C^2)/\cal{T}_z(V)$ of the tangent space of
$\C^2=\C\P^2\setminus\left\{z_0=0\right\}$ by the tangent subspace of $V$.\\
\indent
Then for the unit disk $\D=\left\{t\in\C: |t|<1\right\}$ we define the holomorphic map
\begin{equation*}
\phi:V\times \D\to \C^2
\end{equation*}
by the formula $\phi(z,t)=z+t\cdot n(z)$.
By the inverse function theorem $\phi$ is biholomorphic in the neighborhood
$V\times \D(\epsilon)$ for some $\epsilon>0$. Therefore, function
\begin{equation*}
{\tilde f}(\zeta_1,\zeta_2)=f(\psi(\phi^{-1}(\zeta))),
\end{equation*}
where $\psi(\zeta,t)=\zeta$, is a holomorphic function on a small enough neighborhood
of $V$ in $\C\P^2$ satisfying ${\tilde f}\big|_{V}=f$. Then defining
\begin{equation}\label{g-Definition}
g(z_0,z_1,z_2)=z_0^{-1}{\tilde f}(z_1/z_0,z_2/z_0)
\end{equation}
we obtain $g$ satisfying \eqref{gEquality}.
\end{proof}
\indent
Let $f$ be the holomorphic function defined in \eqref{HolomorphicFunction}.
We consider the holomorphic function $g(z)$ on $U^{\epsilon}\subset \*S^{5}(1)$
of negative homogeneity constructed in \eqref{g-Definition},
and the following integral representation of this function on $V$
\begin{equation}\label{gValues}
g(w^{(k)})=\frac{1}{d+1}\cdot\frac{\det A_k[g](w,w^{(1)},\dots,w^{(d-1)})}{\det A(w)},
\end{equation}
obtained in \cite{P} under conditions \eqref{TwoConditions} with $A(w), A_k[g]$
defined in \eqref{Vandermonde} and \eqref{GValues}. For reader's convenience we provide below
a copy of this theorem.

\newtheorem{thm}{Theorem from \cite{P}}
\renewcommand{\thethm}{}
\begin{thm}\label{HolomorphicTheorem}
Let $V\subset \left\{\C\P^2\setminus\left\{z_0=0\right\}\right\}$
and $U^{\epsilon}$ be as in \eqref{AffineCurve} and \eqref{Uepsilon} respectively,
and let $g$ be a holomorphic function of negative homogeneity in $U^{\epsilon}$.
Let $z\in V$ be a fixed point, and let ${\cal V}_z\ni z$
be a neighborhood of $z$ in $V$, such that conditions \eqref{TwoConditions} are
satisfied.\\
\indent
Then for $w\in {\cal V}_z$ equality \eqref{gValues} holds for the values of $g$ at the points
of ${\cal S}(w)$, where ${\cal S}(w)$ is defined in \eqref{RCondition}, and $A(w)$, $A_k[g]$
are defined in \eqref{Vandermonde}, \eqref{GValues}, and \eqref{Gamma}.\begin{flushright}$\square$\end{flushright}
\end{thm}
\textit{Proof of Theorem~\ref{Main}.}
Using integral representation \eqref{gValues} from Theorem above for function $g=z_0^{-1}\cdot f$
from \eqref{g-Definition}
we obtain the following representation for function $u$ at the points of ${\cal S}(w)$:
\begin{equation}\label{CommentsuValues}
u(w^{(k)})=\frac{1}{d+1}\cdot\mbox{Re}
\left\{\frac{w_0^{(k)}\cdot\det A_k[g](w,w^{(1)},\dots,w^{(d-1)})}{\det A(w)}\right\}
+\sum_{r=1}^{m-1}a_rh_r(w),
\end{equation}
where the coefficients $\left\{a_r\right\}_{r=1}^m$ are defined in the Proposition~\ref{ZeroPeriods-1}.
\begin{flushright}\qed\end{flushright}
\indent
In conclusion we describe an application of Theorem~\ref{Main}.

\begin{proposition}\label{Application}
Let all conditions of Theorem~\ref{Main} be satisfied, and let $u$ be a harmonic function
on $V$, such that the values of $u$ are known
in some connected neighborhood $U\supset bV$. Then there exists a holomorphic function $f$ in $U$
such that equalities \eqref{uValues} hold for the values of $u$ at the points of ${\cal S}(w)$.
\end{proposition}
\begin{proof}
It suffices to notice that since the values of function $u$ are known in a neighborhood $U\supset bV$, the values
of $\partial u$ can be found in the same neighborhood. Then coefficients $\left\{a_r\right\}_1^{m-1}$ and
the function $h$ in formula \eqref{aCoefficients} can be evaluated in $V$, and the sought function
$f$ is defined in $U$ by formula \eqref{HolomorphicFunction}.
\end{proof}

\end{document}